\documentclass[11pt,letterpaper]{amsart}
\usepackage{mathtools}
\usepackage{mathrsfs}
\usepackage[OT2, T1]{fontenc}
\usepackage{url}
\usepackage{amsmath}
\usepackage{array}
\usepackage{graphicx}
\usepackage{amsfonts}
\usepackage{amssymb}
\usepackage{amstext}
\usepackage{amsthm}
\usepackage{hyperref}
\usepackage{enumerate}

\usepackage{accents}

\hypersetup{
    colorlinks=true,
    linkcolor=blue,
    citecolor=blue,
    urlcolor=teal
}

\usepackage{colonequals}
\usepackage{enumitem}
\usepackage[all,cmtip]{xy}

\usepackage{tikz} 
\usepackage{tikz-cd}

\newcommand{\set}[1]{\left\{ #1 \right\}}
\newcommand{\va}[1]{\left| #1 \right|}

\newcommand{\sm}{\setminus}

\newcommand{\CC}{\mathbb{C}}

\newcommand{\OO}{\mathcal{O}}

\newcommand{\QQ}{\mathbb{Q}}
\newcommand{\PP}{\mathbb{P}}

\newcommand{\et}{\textnormal{\'et}}

\newcommand{\RR}{\mathbb{R}}

\newcommand{\ZZ}{\mathbb{Z}}

\newcommand{\Pic}{\textnormal{Pic}}

\newcommand{\Ker}{\textnormal{Ker}}

\newcommand{\wt}[1]{\widetilde{#1}}

\newcommand{\ca}[1]{{\mathcal{#1}}}
\newcommand{\bb}[1]{{\mathbb{#1}}}

\newcommand{\mr}[1]{{\mathscr{#1}}}

\let\rm\relax 
\newcommand{\rm}[1]{{\mathrm{#1}}}

\DeclareSymbolFont{bbm}{U}{bbm}{m}{n}
\DeclareSymbolFontAlphabet{\mathbbm}{bbm}

\numberwithin{equation}{section}

\newtheorem{theorem}{Theorem}[section]
\newtheorem{lemma}[theorem]{Lemma}

\newtheorem{proposition}[theorem]{Proposition}
\newtheorem{corollary}[theorem]{Corollary}

\theoremstyle{definition}
\newtheorem{definition}[theorem]{Definition}

\newtheorem{remark}[theorem]{Remark}

\newtheorem{notation}[theorem]{Notation}

\theoremstyle{remark}

\newcommand{\F}{\mathbb{F}}

\newcommand{\Q}{\mathbb{Q}}
\newcommand{\Qbar}{\overline{\Q}}
\newcommand{\Qtbar}{\overline{\Q(t)}}
\newcommand{\Qtbarsmash}{\smash{\overline{\mathbb{Q}(t)}}}

\newcommand{\Z}{\mathbb{Z}}

\newcommand{\Ag}{\mathcal{A}_g}

\newcommand{\Agr}{\mathcal A_{g,[r]}}
\newcommand{\Agdeltar}{\mathcal A_{g,\delta,[r]}}

\DeclareMathOperator{\GL}{GL}
\DeclareMathOperator{\SL}{SL}
\DeclareMathOperator{\GSp}{GSp}
\DeclareMathOperator{\Sp}{Sp}

\DeclareMathOperator{\Gal}{Gal}
\DeclareMathOperator{\End}{End}
\DeclareMathOperator{\Hom}{Hom}

\usepackage[backend=bibtex,style=alphabetic]{biblatex}

\begin{document}
\title[Powers of abelian varieties not isogenous to a Jacobian]{Powers of abelian varieties over $\overline{\mathbb{Q}(t)}$ not isogenous to a Jacobian}

\author{Olivier de Gaay Fortman and Ananth N.\ Shankar}

\date{\today}

\maketitle

\vspace{-10mm}

\begin{abstract}
    We prove the existence of abelian varieties over $\Qtbar$ with no power isogenous to a Jacobian. Moreover, given a positive integer $N$, we prove the existence of abelian varieties over $\Qtbar$ with maximal monodromy such that the $n$th power is not isogenous to a Jacobian for $n \leq N$. We make use of an Arakelov inequality established by Lu and Zuo, as well as intersection theoretic methods, to prove our main results. 
\end{abstract}

\section{Introduction}

Every abelian variety $A$ over an algebraically closed field is dominated by a Jacobian, but not always isogenous to one: over $\CC$ this follows from a countability argument, over $\Qbar$ this was proven by Chai--Oort and Tsimerman \cite{chaioort, tsimerman-jacobian}, and over $\overline{\F_p(t)}$ this was proven 
by the second named author and Tsimerman  
\cite{shankartsimerman-Fptbar}.\footnote{The case of finite fields is expected to be different as suggested by heuristics, independently offered by \cite{shankar-tsimerman-heuristics} and Poonen.}

This leaves open the question whether some power of $A$ is isogenous to a Jacobian. 
In \cite{dGFSchreieder-2025}, the first named author and Schreieder 
showed that the answer is again no in general: for any $g \geq 4$, 
there exists a 
complex abelian variety of dimension $g$ with no power isogenous to a Jacobian, and for $g=5$ one can find intermediate Jacobians of smooth cubic threefolds among such examples. The authors ask whether one can find  
abelian varieties 
over $\Qbar$ with no power isogenous to a Jacobian, see \cite[Remark 1.6]{dGFSchreieder-2025}. 

The first main result of this paper says that if $g \geq 16$ is a power of $2$, then such examples exist at least over $\smash{\Qtbar}$. By a Jacobian, we mean the Jacobian of a projective stable curve of compact type. 

\begin{theorem} \label{theorem:special}
    Let $g=2^d$ with $d \geq 4$. Then there exists an abelian variety $A$ of dimension $g$ over $\overline{\mathbb{Q}(t)}$ with no power $A^n$ isogenous to a Jacobian. 
\end{theorem}

Let $\Ag$ be the moduli space of principally polarized abelian varieties of dimension $g$. A point $x \in \Ag$ is called \emph{Hodge generic} if $x$ is not contained in a special subvariety of $\Ag$ of positive codimension---equivalently, the associated Mumford--Tate group is $\GSp_{2g}$. 
In \cite{tsimerman-jacobian}, Tsimerman showed that one can find CM abelian varieties of dimension $g \geq 4$ which are not isogenous to a Jacobian; in \cite{masser-zannier} and \cite{shankartsimerman-Fptbar}, the authors (using different methods) showed that there are also Hodge generic $\Qbar$ points of $\Ag$ with that property. 

As the proof of Theorem \ref{theorem:special} shows, for $g = \smash{ 2^d} \geq 16$, there are $x\in \smash{\Ag(\Qtbar)}$, the geometric generic point of a special curve $C \subset \Ag$, such that $A_x$ has no power isogenous to a Jacobian. It is natural to ask whether there exist Hodge generic points $x \in \smash{\Ag(\Qtbar)}$ with that same property. Our second main result says that at least for bounded powers, the answer is yes.

\begin{theorem} \label{theorem:generic}
Let $g = 2^d \geq 16$ and $N \geq 1$. Then there exists a Hodge generic point $x = [A] \in \Ag(\Qtbar)$ such that the corresponding abelian variety $A$ over $\Qtbar$ has no power $A^n$ with $n \leq N$ which is isogenous to a Jacobian.
\end{theorem}

\subsection{Outline of proof}
We now briefly describe our methods. Let $\Agr$ be the moduli space of principally polarized abelian varieties of dimension $g$ with symplectic level $r\geq 3$ structure. The proof of Theorem \ref{theorem:special} uses an Arakelov inequality due to Lu and Zuo \cite[Theorem 1.4]{luzuo-2019}, which says in particular that given a smooth proper curve $C \subset \Agr$ with induced abelian scheme $\pi \colon A \to C$, one has $\deg(\pi_\ast \Omega_{A/C}) < (g/2) \cdot \deg(\Omega_C)$ whenever $g \geq 12$ and $C$ is generically contained in the Torelli locus. On the other hand, this becomes an equality if $C$ is a Shimura curve of Mumford type (cf.\ Definition \ref{definition:nonsplit}), from which Lu and Zuo deduce that such Shimura curves cannot be generically contained in the Torelli locus if $g \geq 12$. We apply their Arakelov inequality to show that, more generally, the $n$-th power $A_{\bar \eta}^n$ cannot be isogenous to a Jacobian for every $n$, for $\eta \in C$ the generic point. 

We prove the above for Shimura curves of Mumford type $C \subset \Agr$ which are not necessarily smooth. Such curves exist in $\Agr$ for each $r \geq 3$. For technical reasons, we proceed in the second part of the paper to assume that $r$ is prime. We pick a suitable surface $S\subset\Agr$ that contains $C$ as well as Hodge generic points. Theorem \ref{theorem:special} implies that the induced abelian scheme $B\rightarrow S$ has the property that the $n$-th power $B_{\bar \eta}^n$ is not isogenous to a Jacobian for every $n$, where now $\eta \in S$ denotes the generic point of $S$. Bounding $n$, we then use an intersection-theoretic argument to find a curve $C_{\textrm{gen}} \subset S$ over $\Qbar$ that satisfies the conclusion of Theorem \ref{theorem:generic}. 

\subsection{Acknowledgements}

We would like to thank Emiliano Ambrosi, Philip Engel, Stefan Schreieder and Jacob Tsimerman for stimulating discussions. 

O.d.G.F.~has received funding from the ERC Consolidator Grant FourSurf \textnumero 101087365. A.S.~was partially supported by the NSF grant DMS-2338942, the Institute for advanced studies (via the NSF grant DMS-2424441), and a Sloan research fellowship.

\section{Notation and conventions} \label{conventions}

We work in characteristic zero. 
For an integer $r \geq 3$, 
we denote by $\ca A_{g,[r]}$ the moduli scheme of $g$-dimensional principally polarized abelian varieties with symplectic level $r$ structure, and by $\ca X \to \Agr$ the universal family. 
For an abelian scheme $f \colon A \to S$ over a smooth variety $S$, the \emph{Hodge line bundle} is the line bundle $\det(f_\ast \Omega_{A/S})$, which we often simply denote by $L = \det(f_\ast \Omega_{A/S})$. If $S$ is a proper curve, then the degree of the vector bundle $f_\ast \Omega_{A/S}$ is the degree of its determinant $L$. A \emph{Shimura curve} $C \subset \Agr$ will be an irreducible component of a one-dimensional Shimura subvariety of $\Agr$.
A subvariety $Z \subset \Agr$ is said to have \emph{maximal monodromy} 
if the map $\pi_1^{\textnormal{\'et}}(Z) \to \pi_1^{\textnormal{\'et}}(\Agr)$ is surjective. 

\section{Shimura curves of Mumford type}

In this section, we explain the definition of Shimura curves of Mumford type, which are certain special curves $C \subset \ca A_{g,[r]}$. 
We follow the construction of \cite[Section 5]{viehwegzuo-2004}, see also \cite[p.\ 52]{luzuo-2019}.  

Let $d > 1$ be an integer. 
Let $F$ be a totally real field of degree $d$ over $\QQ$. 
Let $A$ be a quaternion algebra over $F$, which is which at one real place of $F$ and ramified at its other real places.
Let $D = \rm{Cor}_{F/\QQ}(A)$ be the corestriction of $A$, which is a central simple algebra over $\QQ$. 
That is, put $D_1 = \otimes_\sigma \left(A \otimes_{F, \sigma}\Qbar \right)$, where $\sigma$ ranges over $\Hom(F, \Qbar)$, and then $D = (D_1)^{\Gal(\Qbar/\QQ)}$. 

\begin{lemma}\label{lemma:cases}
Let $F$ and $A$ be as above. 
Then either one of the following conditions holds:
\begin{enumerate}
\item \label{case:i} We have $D \simeq \rm{M}_{2^d}(\QQ)$ and $d$ is odd. 
    \item \label{case:ii} We have $D \simeq \rm{M}_{2^{d-1}}(B)$ for some non-split quaternion algebra $B$ over $\QQ$, which splits over a quadratic extension $L = \QQ(\sqrt b)$ of $\QQ$ which is imaginary if and only if $d$ is even. 
\end{enumerate}
\end{lemma}

\begin{proof}
This is shown in 
 \cite[Lemma 5.7]{viehwegzuo-2004}. Alternatively, see \cite[Proposition 4.2, item (3)]{moonen-classification}. 
 \end{proof}
\begin{remark} \label{remark}
Let $d > 1$ be an integer, let $F$ be a totally real number field of degree $d$, and let $\Sigma$ be the set of places of $F$. Let $A$ be a quaternion algebra which is split at one real place of $F$ and ramified at its other real places. Such an algebra exists because for every finite set $R \subset \Sigma$ of non-complex places such that $\va{R}$ is even, there exists a quaternion algebra $A$ over $F$ which ramifies exactly at the places $v \in R$ (see \cite[Theorem 14.6.1]{voight}). If $d$ is even then
 $D = \rm{Cor}_{F/\QQ}(A) \simeq \rm{M}_{2^{d-1}}(B)$ for some non-split quaternion algebra, see Lemma \ref{lemma:cases}. 
If $d$ is odd, then one can choose $A$ such that $D \simeq \rm{M}_{2^d}(\QQ)$ (resp.\ $D \simeq \rm{M}_{2^{d-1}}(B)$ for some non-split $B$). Indeed, this is true because corestriction of Brauer groups along a finite morphism of schemes commutes with base change (see \cite[Proposition 3.8.1]{brauer}), and the corestriction map of Brauer groups attached to a finite extension of local fields commutes with the local invariant (see \cite[p.\ 237, Theorem 10]{Lorenz2008}). 
\end{remark}

Let again $d > 1$, $F$ a totally real field of degree $d$ over $\QQ$, $A$ a quaternion algebra over $F$ which is split at one real place of $F$ and ramified at its other real places, and $D = \rm{Cor}_{F/\QQ}(A)$. 
By Lemma \ref{lemma:cases}, there is an embedding
\begin{align}\label{embedding}
D \hookrightarrow \rm{M}_{2^{d+1}}(\QQ), \quad \quad \epsilon \in \set{0,1}. 
\end{align}
Let $\OO \subset A$ be an order in $A$, and let $\OO^1 \subset \OO^\ast$ be the subgroup of units that have reduced norm one. Let $\sigma \colon F \to \RR$ be such that $A \otimes_{F,\sigma} \RR \simeq \rm{M}_2(\RR)$; then the image of $\OO^1 \subset A^\ast$ in $\GL_2(\RR) \subset \rm{M}_2(\RR)$ under this isomorphism is contained in $\SL_2(\RR)$, which gives a natural embedding $\OO^1 \hookrightarrow \SL_2(\RR)$. Let $\Gamma \subset \OO^1$ be a finite index torsion-free arithmetic subgroup. Then $C_1 \coloneqq \Gamma \sm \bb H$ is a smooth proper curve, see \cite[Chapter 9]{shimura}. The diagonal embedding $\OO \to \otimes_{\sigma \in \Hom(F, \overline{\QQ})} \left(A \otimes_{F, \sigma}\Qbar \right)$ induces an embedding $\OO^1 \to D^\ast$, the embedding \eqref{embedding} induces an embedding $D^\ast \hookrightarrow \GL_{2^{d+\epsilon}}(\QQ)$, 
and the composition 
\[
    \pi_1^\et(C_1) = \Gamma \hookrightarrow \OO^1 \hookrightarrow D^\ast \hookrightarrow \GL_{2^{d+\epsilon}}(\QQ), \quad \quad \epsilon \in \set{0,1},  
\]
corresponds to a local system  on the curve $C_1$, which, by \cite[Lemma 5.10]{viehwegzuo-2004}, 
underlies a family of $g$-dimensional polarized abelian varieties $$B_1 \to C_1, \qquad g = 2^{d + \epsilon -1}.$$ 

\begin{definition}\label{definition:nonsplit}
 Let $r \geq 3$ be an integer. A \emph{Shimura curve of Mumford type} in $\Agr$ (see \cite{viehwegzuo-2004}) is a proper Shimura curve $C  \subset \Agr$ of the following form. Consider a polarized abelian scheme $B_1 \to C_1$ as defined above. 
Pick an étale cover $C_2 \to C_1$ with pull-back $B_2 \to C_2$, so that there is a principally polarized abelian scheme $A_2 \to C_2$ isogenous to the polarized abelian scheme $B_2 \to C_2$, such that $A_2 \to C_2$ carries a symplectic level $r$ structure. This gives a morphism $C_2 \to \ca A_{g,[r]}$, and we let $C \subset \ca A_{g,[r]}$ denote its image. 

Let us say that the Shimura curve of Mumford type $C$ is \emph{split} (resp.\ \emph{non-split}) if we are in case \eqref{case:i} (resp.\ \eqref
{case:ii}) of Lemma \ref{lemma:cases}. 
\end{definition}
 
In dimension $g=4$, the construction of such curves is due to Mumford (cf.~\cite{mumford}), whence the name. 

\begin{remark} \label{remark:deligne}
Let $r \geq 3$ be an integer and let $C \subset \Agr$ be a Shimura curve of Mumford type. Then there exists an integer $s \geq 3$ with $r \mid s$ and a Shimura curve of Mumford type $C' \subset \ca A_{g,[s]}$ which is smooth, and such that $C'$ dominates $C$ with respect to the finite étale cover $\ca A_{g,[s]} \to \Agr$. Indeed, this follows from \cite[Proposition 1.15]{deligne-travauxshimura-1971} (compare \cite[Lemma B.6]{dGFSchreieder-2025}). 
\end{remark}

Let $C \subset \ca A_{g,[r]}$ be a Shimura curve of Mumford type. Consider the corresponding family of $g$-dimensional 
principally polarized abelian varieties 
\begin{align} \label{mumfordtype}
f \colon A \to C, \qquad \text{where} \qquad g = \begin{cases}
			2^{d-1} & \text{in case \eqref{case:i},}\\
            2^d & \text{in case \eqref{case:ii}.}
		 \end{cases}
\end{align}
Then, by \cite[Lemma 5.9]{viehwegzuo-2004}, we have for the generic fiber $A_{\eta}$ of $f$, that
{\small
\begin{align} \label{End}
\End(A_\eta)\otimes \QQ = 
\begin{cases}
			\;\;\QQ & \text{in case \eqref{case:i};}\\
            \begin{cases}\text{a totally definite (resp.~indefinite)} & \\
            \text{quaternion algebra over $\QQ$ if $d$} &
             \\
            \text{is even (resp.\ odd)}
& 
\end{cases} & \text{in case \eqref{case:ii}}.
		 \end{cases}
\end{align}
}
\begin{lemma} \label{simple}
In the above notation, the geometric generic fiber $A_{\bar \eta}$ is simple. More precisely, $\End(A_{\bar \eta}) \otimes \Q = \End(A_\eta) \otimes \Q$, and this is given by \eqref{End}.
\end{lemma}
\begin{proof}
Let $\Gamma'\subset \Gamma$ be a finite index subgroup. It suffices to prove that the representation of $\Gamma'$ induced by the Betti local system is irreducible. This follows from the general fact that if $G$ is a $\QQ$-algebraic group and $V$ a representation of $G$ such that $V_\mathbb{C}$ is irreducible, then $V_\mathbb{C}$ stays irreducible when restricted to any Zariski-dense subgroup of $G(\Q)$.  
\end{proof}

\section{Abelian schemes over special curves with no power isogenous to a Jacobian}

The goal of this section is to prove our first main theorem, i.e.\  Theorem \ref{theorem:special} from the introduction. The candidate for the abelian variety over $\Qtbar$ that will turn out not to admit any power isogenous to a Jacobian, will be given by the geometric generic fiber of the natural abelian scheme over a Shimura curve of Mumford type $C \subset \Agr$ defined in the previous section. 

\subsection{Generalities on Hodge line bundles of abelian schemes}

\begin{lemma} \label{lemma:hodgebundle}
    Let $S$ be a smooth variety over $\CC$.   
    \begin{enumerate}
        \item \label{flatbasechange:1} Let $\pi \colon S'\to S$ a morphism of smooth varieties. Let $f \colon A \to S$ be an abelian scheme with base change $f' \colon A' \to S'$. Define $L = \det(f_\ast \Omega_{A/S})$ and $L' = \det(f'_\ast \Omega_{A'/S'})$. Then $\pi^\ast(L) \simeq L'$. 
        \item \label{flatbasechange:2} Let $\phi \colon A_1 \to A_2$ be an isogeny of abelian schemes $f_i \colon A_i \to S$. Let $L_i = \det({f_i}_\ast \Omega_{A_i/S})$. There is a canonical isomorphism $L_1\simeq L_2.$ 
    \end{enumerate}
\end{lemma}
\begin{proof}
This is well-known, and can e.g.\ be deduced from the canonical isomorphism $f_\ast \Omega_{A/S} \simeq e^\ast \Omega_{A/S}$ where $e \colon S \to A$ denotes the origin section. 
\end{proof}

Let $f \colon A \to C$ be an abelian scheme of relative dimension $g$ over a smooth proper curve $C$, with Hodge line bundle $L  = \det(f_\ast \Omega_{A/C})$. We have $\deg(L) \leq (g/2) \cdot \deg(\Omega_C)$ by  \cite{faltings}.  Let $\pi \colon C' \to C$ be a finite morphism of smooth proper curves, of degree $d \geq 1$ and ramification divisor $R \subset C'$. Let $A' \to C'$ denote the base change of $A\to C$ along $\pi$, and define $L' = \det(f'_\ast \Omega_{A'/C'})$ as the Hodge line bundle of $A' \to C'$. 

\begin{lemma} \label{lemma:inequality}In the above situation, 
let $m \in \ZZ_{\geq 0}$ such that 
 $\deg(L)+m=(g/2) \cdot \deg( \Omega_C)$. Then $\deg(L') + d\cdot m = (g/2) \cdot  (\deg(\Omega_{C'}) - \deg(R))$. 
\end{lemma}
\begin{proof}
By Lemma \ref{lemma:hodgebundle}, we have  $\deg(L') = d \cdot \deg(L)$. 
Moreover, we have $\deg(\Omega_C’) = d \cdot \deg(\Omega_C)+\deg(R)$. 
From these two equalities, we conclude that
$
\deg(L') + d\cdot m  = d \cdot (g/2)\cdot \Omega_C = (g/2)\cdot \left(\deg(\Omega_{C'}) - \deg(R) \right)$. 
\end{proof}

\begin{lemma} \label{lemma:higgs-power}
Let $Z$ be a smooth variety over $\CC$. Let $f \colon A \to Z$ be an abelian scheme, and $g \colon A^n \to Z$ its $n$-th self fiber product over $Z$. Let $h \colon Q \to Z$ be an abelian scheme and $\phi \colon A^n \to Q$ an isogeny of abelian schemes. Define the associated Hodge line bundles as $L_f = \det(f_\ast \Omega_{A/Z})$, $L_g = \det(g_\ast \Omega_{A^n/Z})$ and $L_h = \det(h_\ast \Omega_{Q/Z})$. Then the following assertions are true. 
\begin{enumerate}
\item We have $L_h \simeq L_g \simeq (L_f)^{\otimes n}$. 
\item Assume that $Z$ is a curve and 
 that $\deg(L_f) = (g/2) \cdot \deg(\Omega_Z)$. Then $\deg(L_h) = (ng/2) \cdot \deg(\Omega_Z)$. 
\end{enumerate}
\end{lemma}
\begin{proof}
The first item follows from Lemma \ref{lemma:hodgebundle} and the canonical isomorphism $g_\ast \Omega_{A^n/Z} \simeq (f_\ast \Omega_{A/Z})^{\oplus n}$.  
If $Z$ is a curve and $\deg(L_f) = (g/2)\cdot \deg(\Omega_Z)$, then it follows that $\deg(L_h) = \deg(L_g) = n \cdot \deg(L_f)= (ng)/2 \cdot \deg(\Omega_Z)$. 
\end{proof}

\subsection{Arakelov bound for Shimura curves}
We recall the following result:

\begin{theorem}[M\"oller]
Let $C \subset \Agr$ be a smooth irreducible component of a one-dimensional Shimura subvariety of $\Agr$. Then the
corresponding abelian scheme $f \colon A \to C$ reaches the Arakelov bound. In particular, if $C$ is proper and the local system $R^1f_\ast \QQ$ has no non-trivial unitary part, then 
\begin{align} \label{higgs-equality}
\deg(L) = (g/2)\cdot \deg(\Omega_C), \quad \quad L = \det(f_\ast \Omega_{A/C}). 
\end{align}
\end{theorem}
\begin{proof}
See \cite[Theorem 1.2]{moller-shimura-teichmuller}.  
\end{proof}

\begin{corollary} \label{corollary:equality}
Any smooth Shimura curve of Mumford type $C \subset \Agr$ satisfies the equality \eqref{higgs-equality}. 
\end{corollary}
\begin{proof}
    Such a curve is smooth and proper, and the natural abelian scheme $f \colon A \to C$ has no non-trivial unitary part in $R^1f_\ast \QQ$, because the monodromy is irreducible and non-compact. 
\end{proof}

\subsection{Proof of the theorem for special curves}

We have the following result, of which Theorem \ref{theorem:special} is a direct consequence (as we show below).

\begin{theorem}\label{theorem:shimuracurve}
Let $g=2^d$ with $d \geq 4$. Let $r \geq 3$ and let $C \subset \Agr$ be a proper curve which is a Shimura curve of Mumford type, see Definition \ref{definition:nonsplit}. Let $A \to C$ be the induced abelian scheme. Then no power of the geometric generic fiber $A_{\bar \eta}$ is isogenous to the Jacobian of a stable curve. 
\end{theorem}

\begin{proof}
Assume for a contradiction that for some $n \geq 1$, there exists a stable compact type curve $X  = \cup_i X_i$ over $\smash{\overline{\mathbb{Q}(t)}}$, with irreducible components $X_1, \dotsc, X_k$ (which are smooth projective connected curves over $\smash{\overline{\mathbb{Q}(t)}}$), such that $A_{\bar \eta}^n$ is isogenous to the Jacobian $JX = JX_1 \times \cdots \times JX_k$. By Lemma \ref{simple}, the abelian variety $A_{\bar \eta}$ is simple. Hence, by Poincaré irreducibility, we have that $A_{\bar \eta}^{n'}$ is isogenous to $JX_1$ for some $n' \leq n$. In other words, to arrive at a contradiction, we may assume that $k=1$, and that there exists an isogeny \begin{align} \label{isogeny}A_{\bar \eta}^n \to JX \end{align} between $A_{\bar \eta}^n$ and the Jacobian $JX$ of a smooth projective connected curve $X$ over $\smash{\overline{\mathbb{Q}(t)}}$. 
In view of Remark \ref{remark:deligne}, 
up to replacing $\ca A_{g,[r]}$ by some finite étale cover $\ca A_{g,[r']} \to \ca A_{g,[r]}$ with $r \mid r'$ and $C$ by an irreducible component of its preimage, we may assume that $C$ is smooth, and that the kernel of \eqref{isogeny} is the geometric generic fiber $M_{\bar \eta} \subset A_{\bar \eta}^n[r]$ of a subgroup scheme $M \subset A^n[r]$ which is finite étale over $C$. 
Taking the quotient by $M$ yields an abelian scheme $A^n/M \to C$ such that $(A^n/M)_{\bar \eta} = JX$. 

Consider the polarization on $A^n_{\bar \eta}$ pulled back from the canonical polarization on $JX$ under \eqref{isogeny}. As $A_{\bar \eta}$ is principally polarized, we can view this polarization as an $n \times n$-matrix with values in $\End(A_{\bar \eta})$. As $\End(A_{\bar \eta}) = \End(A_\eta)$ by Lemma \ref{simple}, this yields a polarization on $A_\eta^n$ which extends to a polarization on the abelian scheme $A^n \to C$, which in turn descends to a principal polarization on $A^n/M \to C$ as it does so on the geometric generic fiber. We then let $\tilde  C \to C$ be a finite étale cover so that if $\tilde  M \subset \tilde  A \to \tilde  C$ denotes the pull-back of $M \subset A \to C$, then the principally polarized abelian scheme $\tilde  Q \coloneqq \tilde  A^n/\tilde  M \to \tilde  C$ admits a sympletic level $s$ structure, for some $s \geq 3$. This gives a finite morphism $\tilde  C \to \ca A_{ng,[s]}$, and we let $C' \subset \ca A_{ng,[s]}$ denote its image. Define $Q' \to C'$ as the induced abelian scheme over $C'$.

Note that $C' \subset \ca A_{ng,[s]}$ is a special subvariety. Hence, by \cite[Lemma B.6]{dGFSchreieder-2025}, there exists a finite étale cover $\ca A_{ng,[s']} \to \ca A_{ng,[s]}$ with $s \mid s'$ and a smooth special subvariety $C'' \subset \ca A_{ng,[s']}$ which is an irreducible component of the preimage of $C$ in $\ca A_{ng,[s']}$. Up to replacing $\tilde  C$ by a finite étale cover, we may thus assume that $C'$ is smooth. 

The situation may be depicted in the following diagram, in which the three curves $C', \tilde C$ and $C$ are smooth, 
the morphism $\tilde  C \to C'$ is finite, the morphism $\tilde  C \to C$ is finite étale, 
and the two squares are cartesian:
\[
\xymatrix{
&&\tilde  Q\ar[dl] \ar[dr]&&\tilde  A^n \ar[dr]\ar[ll] \ar[dl]& &\\
&Q'\ar[dr]&& \tilde  C\ar[dl] \ar[dr] && A^n \ar[dl] &
\\
 \ca A_{ng,[s]}&& C' \ar@{_{(}->}[ll] & & C \ar@{^{(}->}[rr]& &\ca A_{g,[r]}.
}
\]
Consider the abelian schemes $f \colon A \to C, g \colon A^n \to C$, $\tilde h \colon \tilde  Q \to \tilde  C$, and $h' \colon Q' \to C'$, and define $L_f, L_g, L_{\tilde h}$ and $L_{h'}$ as the respective Hodge line bundles of these families. 
As $C \subset \Agr$ is a smooth Shimura curve of Mumford type, we have by Corollary \ref{corollary:equality} that $\deg(L_f) = (g/2) \cdot \deg(\Omega_C)$. 
By Lemmas \ref{lemma:inequality} and \ref{lemma:higgs-power}, we therefore obtain:
\begin{align} \label{first}
\deg(L_{\tilde h}) = \frac{ng}{2} \cdot \deg(\Omega_{\tilde  C}). 
\end{align}
The geometric generic fibers of $\tilde  Q \to \tilde  C$ and $Q' \to C'$ are canonically isomorphic as principally polarized abelian varieties, and by construction isomorphic to the Jacobian $JX$ of the smooth projective connected curve $X$ over $\smash{\overline{\mathbb{Q}(t)}}$. Therefore, the smooth curve $C' \subset \ca A_{ng,[s]}$ is contained in the closed Torelli locus, and intersects the open Torelli locus non-trivially. As $g = 2^d>11$, this implies by \cite[Theorem 1.4]{luzuo-2019} that
\begin{align}\label{align:inequality-proof}
\deg(L_{h'}) < \frac{ng}{2} \cdot \deg(\Omega_{C'}). 
\end{align}
The morphism 
$
\tilde  C \to C'
$
is a finite morphism of smooth curves. Therefore, by Lemma \ref{lemma:inequality}, the strict inequality \eqref{align:inequality-proof} implies that we have a strict inequality
$
\deg(L_{\tilde h}) < (ng/2) \cdot \deg(\Omega_{\tilde  C})$, which contradicts \eqref{first}.
\end{proof}

\begin{proof}[Proof of Theorem \ref{theorem:special}]
We must provide an abelian variety of dimension $g = 2^d \geq 16$ over $\smash{\overline{\mathbb{Q}(t)}}$ with no power isogenous to a Jacobian. Let $r \geq 3$ be an integer and $C \subset \Agr$ a non-split Shimura curve of Mumford type (cf.\ Definition \ref{definition:nonsplit}); such a pair $(r,C)$ exists by Remark \ref{remark}. Since $C \subset \Agr$ is a Shimura curve, it is defined over $\Qbar$. Let $A \to C$ be the natural abelian scheme. The geometric generic fiber $A_{\bar \eta}$ is a $g$-dimensional abelian variety over $\smash{\overline{\mathbb{Q}(t)}}$ which, by Theorem \ref{theorem:shimuracurve}, has no power isogenous to a Jacobian. 
\end{proof}

\section{Correspondences for moduli of polarized abelian varieties} \label{sec:correspondences}
In this section, we generalize the classical notion of Hecke correspondence $\Agr \leftarrow \smash{{\ca A}_a} \to \Agr$ attached to an element $a \in \GSp_{2g}(\QQ)_+$, by defining correspondences $\Agr \leftarrow \smash{{\ca A}_a} \rightarrow \ca A_{g,\delta,[r]}$ between $\Agr$ and a moduli space of polarized abelian varieties $\ca A_{g,\delta,[r]}$ of arbitrary polarization type $\delta = (\delta_1, \dotsc, \delta_g)$.

Fix integers $g \geq 1$ and $r \geq 3$, as well as an ordered sequence of positive integers $\delta = (\delta_1, \dotsc, \delta_g)$ with $\delta_i \mid \delta_{i+1}$ for all $i$. Let $\Lambda = \ZZ^{2g}$, with basis $\set{e_1, \dotsc, e_g; f_1, \dotsc, f_g}$. 
Let $\Lambda_\delta 
= \langle e_1, \dotsc, e_g; \delta_1f_1, \dotsc, \delta_gf_g \rangle$. We view $\Lambda$ as a symplectic lattice, with symplectic form $Q$ determined by $Q(e_i, f_j) = \delta_{ij}$ and $Q(e_i, e_j) = Q(f_i, f_j) = 0$, and we view $\Lambda_\delta$ as a non-degenerate alternating lattice, with alternating form defined the restriction of $Q$ to $\Lambda_\delta \subset \Lambda$.

\begin{definition}(\cite[Section 8.3.2]{birkenhakelange}) \label{def:levelstructure}
Let $(A, \lambda_A)$ be a $g$-dimensional polarized complex abelian variety, of polarization type $\delta$.  A \emph{symplectic level $r$ structure} is an isomorphism $\bar \phi \colon A[r] \simeq (\ZZ/r)^{2g}$ that lifts to a symplectic isomorphism $\phi \colon H_1(A,\ZZ) \simeq \Lambda_\delta$. 
\end{definition}
\begin{remark} \label{remark:reduction}
As the reduction map $\Sp_{2g}(\ZZ) \to \Sp_{2g}(\ZZ/r)$ is surjective, a symplectic level $r$ structure on a principally polarized complex abelian variety $(A, \lambda_A)$ is an isomorphism of symplectic $\ZZ/r$-modules $A[r] \simeq (\ZZ/r)^{2g}$.  
\end{remark}

For $r \geq 3$, there is a fine moduli space $\ca A_{g, \delta, [r]}$ of polarized abelian varieties of dimension $g$ with polarization type $\delta$ equipped with a symplectic level $r$ structure; this follows e.g.\ from \cite[Theorem 7.9]{mumford-GIT-2nd} and the remark below that theorem. 
We recall the complex uniformization of this moduli space. 
Let $\Gamma_\delta \subset \rm{Sp}_{2g}(\QQ)$ be the subgroup of matrices that preserve $\Lambda_\delta \subset \Lambda_\delta \otimes_\ZZ \QQ = \QQ^{2g}$. Define $\Gamma_\delta(r) = \Ker\left( \Gamma_\delta \to \GL_{2g}(\ZZ/r)\right)$, and let $\bb H_g$ be the Siegel upper half space. Then (see e.g.\ \cite[Theorem 8.3.2]{birkenhakelange}), we have:
$$
\ca A_{g, \delta, [r]}(\CC) = \Gamma_\delta(r) \sm \bb H_g. 
$$
We will use this uniformization to construct correspondences between $\ca A_{g, \delta, [r]}$ and $\ca A_{g, [r]} = \ca A_{g,(1, \dotsc, 1), [r]}$. 
For ease of notation, put $\Gamma = \Gamma_{(1, \dotsc, 1)} = \Sp_{2g}(\ZZ)$, so that $\Gamma(r) = \Ker\left(\Sp_{2g}(\ZZ) \to \Sp_{2g}(\ZZ/r) \right)$. 
\begin{definition}\label{def:a}
For $a \in \GSp_{2g}(\QQ)_+$, define ${\ca A}_a$ as the complex variety with
$$
{\ca A}_a(\CC) =  \Gamma_a(r) \sm \bb H_g, \qquad \Gamma_a(r) =  a^{-1}\Gamma(r) a \cap \Gamma_\delta(r)  \subset \GSp_{2g}(\QQ).  
$$
Define a correspondence consisting of two finite étale maps $p_1$ and $p_2$,
\begin{align} \label{correspondence}
\xymatrix{
\ca A_{g, [r]} & {\ca A}_a \ar[r]^-{p_2} \ar[l]_-{p_1}& \ca A_{g, \delta, [r]},
}
\end{align}
where $p_1(\Gamma_a(r) \cdot x) = \Gamma(r) \cdot ax$ and $p_2(\Gamma_a(r)\cdot x) = \Gamma_\delta(r)\cdot x$.  
\end{definition}

View $\bb H_g$ as the open subset of the Grassmannian $\rm{Grass}_\CC(g, \Lambda \otimes_\ZZ \CC)$, $\Lambda = \ZZ^{2g}$, consisting of those $g$-dimensional complex vector subspaces $V \subset \Lambda \otimes_\ZZ \CC$ that satisfy $V \oplus \overline V= \Lambda \otimes_\ZZ \CC$ as well as the Riemann bilinear relations with respect to $Q$. Then for a point $x \in \bb H_g$ associated to a subspace $V \subset \Lambda \otimes_\ZZ \CC$, the image of $x \in \bb H_g$ in $\ca A_{g, \delta, [r]}(\CC) = \Gamma_\delta(r) \sm \bb H_g$ corresponds to the abelian variety $A = V/\Lambda_\delta$, polarized by $Q \colon \Lambda_\delta \times \Lambda_\delta \to \ZZ$, with symplectic level $r$ structure induced by the canonical identification $H_1(A,\ZZ) = \Lambda_\delta$. Similarly, the image of $ax \in \bb H_g$ in $ \in \ca A_{g,[r]}(\CC) = \Gamma(r) \sm \bb H_g$ corresponds to the abelian variety $B =aV/\Lambda$, principally polarized by $Q \colon \Lambda \times \Lambda \to \ZZ$, with symplectic level $r$ structure induced by $H_1(B,\ZZ) = \Lambda$. In particular, if we define $B'$ as the abelian variety $B' = V/\Lambda$, principally polarized by $Q \colon \Lambda \times \Lambda \to \ZZ$, then by multiplying $a \in \GSp_{2g}(\QQ)_+$ by a large enough integer, we obtain isogenies
\[
\xymatrix{
A = V/\Lambda_\delta \ar[r]^-f & B' = V/\Lambda \ar[r]^-{g} & B = aV/\Lambda
}
\]
such that $g^\ast(\lambda_B) = m \cdot \lambda_{B'}$ for some $m \geq 1$, and such that $f^\ast(\lambda_{B'}) = \lambda_A$. Consequently, the composition $h = g\circ f$ satisfies $h^\ast(\lambda_B) = m \cdot \lambda_A$. 

\begin{lemma} \label{quadrilateral}
Consider  \eqref{correspondence}. 
Then $p_{1}$ (resp.\ $p_2$) pulls back the Hodge line bundle on $\ca A_{g,[r]}$ (resp.\ $\ca A_{g, \delta, [r]}$) to the Hodge line bundle on ${\ca A}_a$.    
\end{lemma}
\begin{proof}
This follows from the above considerations and Lemma \ref{lemma:hodgebundle}.
\end{proof}

\begin{definition} \label{tauK}
    Consider the correspondence \eqref{correspondence}. For a subvariety $Z \subset \ca A_{g,\delta, [r]}$, define the \emph{Hecke translate} of $Z$ by $a \in \GSp_{2g}(\QQ)_+$  as the equidimensional closed subscheme
    \[
    \tau_a(Z) = p_1(p_2^{-1}(Z)) \subset \ca A_{g, [r]}. 
    \]
\end{definition}

When $\delta = (1, \dotsc, 1)$, the correspondence \eqref{correspondence} is a classical Hecke correspondence on $\ca A_{g, [r]}$, so in that case, Definition \ref{tauK} reduces to the standard definition of a Hecke translates of a subvariety of $\ca A_{g,[r]}$. 

\section{Morphisms between moduli spaces of
polarized abelian varieties attached to symmetric bilinear forms}
\label{section:morphisms}
In this section, we consider a positive definite integral symmetric bilinear form $\alpha \in \rm{M}_n(\ZZ)$, and use it to define a morphism of moduli spaces $\Phi \colon \Agr \to \ca A_{g,\delta,[r]}$. We study some properties of the morphism $\Phi$. 

\begin{notation}\label{notation:alpha}
Let $n$ and $g$ be positive integers. 
Let $\alpha \in \rm{M}_n(\Z)$ be a symmetric matrix whose associated bilinear form on $\Z^n$ is positive definite. For a principally polarized abelian variety $(A, \lambda_A)$ of dimension $g$, we get an isogeny $\alpha \colon A^n \to A^n$, and put $\lambda_\alpha = \lambda_A^n 
\circ \alpha$, where $\lambda_A^n \colon A^n \to (A^n)^\vee$ is the diagonal polarization. 
Then $\lambda_\alpha$ is a polarization on $A^n$ by \cite[Lemma 5.2]{dGFSchreieder-2025}. Let $
\delta = (\delta_1, \dotsc, \delta_{ng})$ with $\delta_i \mid \delta_{i+1}$ 
be the type of this  polarization.  
\end{notation}

As in Section \ref{sec:correspondences}, we put $\Lambda = \ZZ^{2g}$ and let $Q \colon \Lambda \times \Lambda \to \ZZ$ be the standard symplectic form. Define $Q^n \colon \Lambda^n \times \Lambda^n \to \ZZ$ as the diagonal symplectic form induced by $Q$, and put $Q_\alpha(x,y) = Q^n(\alpha x,y)$ for $x,y \in \Lambda^n$. Then $Q_\alpha$ defines a non-degenerate alternating form on $\Lambda^n$. 
Let $Q_\delta \colon \Lambda \times \Lambda \to \ZZ$ be the non-degenerate alternating form defined by the block matrix $\big(\begin{smallmatrix}
  0 & D\\
  -D & 0
\end{smallmatrix}\big)$ where $D = \rm{diag}(\delta_1, \dotsc, \delta_{ng})$, and fix an isomorphism of alternating lattices
\[
f \colon \left(\Lambda, Q_\alpha \right) \xrightarrow{\sim} \left( \Lambda, Q_\delta \right), \qquad \text{with reduction} \qquad \bar f \colon (\ZZ/r)^{2ng} \xrightarrow{\sim} (\ZZ/r)^{2ng}. 
\]
Let $r \geq 3$ be an integer. 
Let $\bar \phi \colon A[r] 
\xrightarrow{\sim} (\ZZ/r)^{2g}$ be a sympletic level $r$ structure of the principally polarized abelian variety $(A, \lambda_A)$, which lifts to a symplectic isomorphism $\phi \colon H_1(A,\ZZ) \xrightarrow{\sim} \Lambda$ (cf.\ Definition \ref{def:levelstructure}). The isomorphism $\phi^n \colon H_1(A^n,\ZZ) = H_1(A,\ZZ)^n \xrightarrow{\sim} \Lambda^n$ identifies the alternating form on $H_1(A^n,\ZZ)$ associated to the polarization $\lambda_\alpha$ with the alternating form $Q_\alpha$ on $\Lambda^n$. If $\bar \phi^n \colon A^n[r] 
\xrightarrow{\sim} (\ZZ/r)^{2ng}$ denotes the induced symplectic level $r$ structure of the principally polarized abelian variety $(A^n, \lambda_A^n)$, then the composition $\bar \phi_\alpha = \bar f \circ \bar \phi^n$ defines a sympletic level $r$ structure $\bar \phi_\alpha \colon A^n[r] \xrightarrow{\sim} (\ZZ/r)^{2ng}$ on the polarized abelian variety $(A^n, \lambda_\alpha)$, in the sense of Definition \ref{def:levelstructure}. This construction yields a morphism of moduli spaces
\begin{align} \label{Phi}\Phi \colon \ca A_{g, [r]} \to \ca A_{ng, \delta, [r]}, \quad \quad (A, \lambda_A, \phi )\mapsto (A^n, \lambda_\alpha, \bar \phi_\alpha).\end{align} 
We establish some properties of the morphism $\Phi$ in the following proposition.  

\begin{proposition} \label{birational} 
Consider the above notation.  Assume $r \geq 3$ is prime.
\begin{enumerate}
\item \label{item:moduli:first}
The morphism $\Phi$ defined in \eqref{Phi} is finite.  
\item \label{item:moduli:second}
Let $Z \subset \ca A_{g,[r]}$ be a subvariety with $\End(\ca X_{\bar \eta}) = \Z$, where $\eta \in Z$ is the generic point and $\ca X \to \ca A_{g,[r]}$ the universal family. 
Let $ Z' = \Phi(Z)$. 
Then every irreducible component of $\Phi^{-1}(Z')$ that dominates $Z'$ is equal to $Z$. Moreover, the morphism $\Phi|_Z \colon Z \to Z'$ is birational; in fact, there is a dense open subset $U' \subset Z'$, with preimage $U = (\Phi|_Z)^{-1}(U') \subset Z$, such that the map $U \to U'$ is an isomorphism. 
\item \label{item:moduli:third} Let $Z \subset \ca A_{g,[r]}$ be as in item \eqref{item:moduli:second}, such that $Z= C$ is a curve. Let $C' = \Phi(C)$. 
Then $\Phi^{-1}(C')$ is the union of $C$ and some closed points. 
\end{enumerate}
\end{proposition}
To prove this proposition, we need two lemmas. 

\begin{lemma} \label{dense}
Let $S$ be an irreducible variety over $\CC$, with generic point $\eta \in S$.  Let $A \to S$ be an abelian scheme. Assume that $\End(A_{\bar \eta}) = \ZZ$. Then the set of $x \in S(\CC)$ such that $\End(A_x) = \ZZ$ is analytically dense in $S(\CC)$. 
\end{lemma}
\begin{proof}
The set of points $x \in S(\CC)$ such that $\End(A_x) \neq \Z$ is a countable union $\cup_iZ_i(\CC) \subsetneq S(\CC)$ where $Z_i \subsetneq S$ are strictly contained algebraic subvarieties. Hence the complement of this set is analytically dense.
\end{proof}

\begin{lemma} \label{lemma:cancelling}
    Let $A$ be an abelian variety, with $\End(A) = \Z$. Let $B$ be another abelian variety and assume $A^n \simeq B^n$ for some $n \geq 1$. Then $A \simeq B$. 
\end{lemma}
\begin{proof}
The hypotheses imply that $A$ and $B$ are isogenous, and there are primitive generators $\phi \in \Hom(A,B)$ and $\psi \in \Hom(B,A)$. The isomorphism $A^n \simeq B^n$ and its inverse $B^n \simeq A^n$ are therefore given by $n\times n$-matrices 
    $
    \left(a_{ij} \phi \right) \in \rm{M}_n(\Hom(A,B))$ and $ \left(b_{ij}\psi\right) \in \rm{M}_n(\Hom(B,A))$, with $a_{ij}, b_{ij} \in \ZZ$. 
    As $\psi \circ \phi = [m] \colon A \to A$ for some $m \in \ZZ$, the composition is of the form $\left(b_{ij}\psi\right) \circ        \left(a_{ij} \phi \right) = \left( [c_{ij} \cdot m] \right)$.   
  This matrix is the identity, thus $m \in \set{\pm 1}$. 
\end{proof}

\begin{proof}[Proof of Proposition \ref{birational}]
To ease notation, write $\ca A_g = \Agr$ and $\ca A_{g, \delta} = \ca A_{g, \delta, [r]}$. 
To prove item \eqref{item:moduli:first}, one first proves that $\Phi$ is proper, which holds by the theory of Néron models and the valuative criterion of properness. 

Next, fix a principally polarized abelian variety with symplectic level $r$ structure $(A, \lambda_A, \bar \phi)$ that defines a point $x = [(A, \lambda_A, \bar \phi)] \in \ca A_g$. The set $\Phi^{-1}(\Phi(x))$ is the set of $y = [(B, \lambda_B, \bar \psi)] \in \ca A_g$ such that $(A^n, \lambda_\alpha, \bar \phi_\alpha) \simeq (B^n, \lambda_B, \bar \psi)$. Any such $B$ is isogenous to $A$, hence $\Phi^{-1}(\Phi(x))$ cannot have irreducible components of positive dimension, which implies that $\Phi^{-1}(\Phi(x))$ is a finite set of points. Thus $\Phi$ is proper and quasi-finite, hence finite. 

Next, we prove  \eqref{item:moduli:second} in the case where $Z = \set{x}$ is a closed point $x = [(A, \lambda_A, \bar \phi_1)] \in \Ag(\CC)$ such that $\End(A) = \ZZ$. We need to show that $\Phi^{-1}(\Phi(x)) = \set{x}$. Let $y = [(B, \lambda_B, \bar \phi_2)] \in \Ag(\CC)$ such that $\Phi(x) = \Phi(y)$. Then $A^n \simeq B^n$, which by Lemma \ref{lemma:cancelling} implies that $A \simeq B$ as abelian varieties; this isomorphism must preserve the principal polarizations because $\End(A) = \Z$. So we may assume $(B, \lambda_B) = (A, \lambda_A)$. 
The isomorphism $(A^n, \lambda_\alpha) \simeq (B, \lambda_B)= (A^n, \lambda_\alpha)$ is then given by a matrix $\tau \in \GL_n(\ZZ)$. Let $\bar \tau \colon A[r]^n \xrightarrow{\sim} A[r]^n$ be the reduction of $\tau$. 
Since $\tau$ compatible with the level structures $\bar \phi_{1,\alpha}$ and $\bar \phi_{2,\alpha}$, it follows that 
$\bar \tau = (\bar \phi_2^n)^{-1} \circ \bar \phi_1^n$, where $\bar \phi_i^n$ is the diagonal symplectic level $r$ structure on $(A^n, \lambda_A^n)$ induced by $\bar \phi_i$. Hence $\bar \tau \in \GL_n(\ZZ/r)$ is a diagonal matrix and the entry $\bar \tau_{ii} \in (\ZZ/r)^\ast$ satisfies $\bar \tau_{ii} \cdot \bar \phi_2 = \bar \phi_1$. But any two symplectic level $r$ structures on a principally polarized abelian variety of dimension $g$ differ by a unique element of $\Sp_{2g}(\ZZ/r)$, 
see Remark \ref{remark:reduction}, hence $\bar \tau_{ii} \cdot \rm{Id}_{2g} \in \Sp_{2g}(\ZZ/r)$. Thus, $(\bar \tau_{ii})^2 = 1$. Since $r$ is a prime number, this implies that 
$\bar \tau_{ii} = \pm 1 \in (\ZZ/r)^\ast$, so that $\bar \phi_2 = \pm \bar \phi_1$ hence $(A, \lambda_A, \bar \phi_1)$ and $(A, \lambda_A, \bar \phi_2)$ are isomorphic principally polarized abelian varieties with symplectic level structure. We conclude that $\Phi^{-1}(\Phi(x)) = \set{x}$.

Next, we prove item \eqref{item:moduli:second} in general. 
Let $Z \subset \Ag$ be as in that item, with $Z'  = \Phi(Z) \subset \ca A_{ng, \delta}$ the image under $\Phi$. We first note that each irreducible component of $\Phi^{-1}(Z')$ that dominates $Z'$ is equal to $Z$, because by the above we have that $\Phi^{-1}(\Phi(x)) = \set{x}$ for any very general point $x \in Z(\CC)$. 

Next, let $\pi \colon \tilde  Z \to Z$ be a strong resolution of singularities (so that, in particular, $\pi^{-1}(Z_{\rm{sm}}) \to Z_{\rm{sm}}$ is an isomorphism, where $Z_{\rm{sm}} \subset Z$ is the smooth locus of $Z$). Let $\Psi \colon \tilde  Z \to Z'$ be the composition of $\tilde  Z \to Z$ with $Z \to Z'$. 
As $\tilde  Z$ and as we work in characteristic zero, we can apply generic smoothness on the target, to deduce the existence of a dense open subset $U' \subset Z'$ such that the morphism $\tilde U \to U'$ is smooth, where $\tilde  U=  
\Psi^{-1}(U')$. Let $U = (\Phi|_Z)^{-1}(U') \subset Z$, and observe that $\tilde U = \pi^{-1}(U)$. 

We claim that $\tilde  U \to U'$ is an isomorphism. 
To prove this, let $\mr S \subset Z(\CC)$ be the set of points $x = [(A, \lambda_A, \bar \phi)] \in Z(\CC)$ such that $\End(A) = \Z$. Then $\mr S$ is dense because $\End(\ca X_{\bar \eta}) = \Z$, see Lemma \ref{dense}, hence $\mr S \cap U \cap Z_{\rm{sm}}  \neq \emptyset$. Let $x = [(A, \lambda_A, \bar \phi_1)] \in \mr S \cap U \cap Z_{\rm{sm}}$. As $\pi^{-1}(Z_{\rm{sm}}) \to Z_{\rm{sm}}$ is an isomorphism, there is a unique preimage $\tilde x \in \tilde  Z$ of $x \in Z$. 
As $\Phi^{-1}(\Phi(x)) = \set{x}$, we have $\Psi^{-1}(\Psi(\tilde x)) = \set{\tilde x}$. This means that the smooth morphism $ \tilde  U \to U'$ has zero-dimensional fibers, so it is smooth and quasi-finite, hence étale. Moreover, it is proper, hence finite étale. The fact that $\Psi^{-1}(\Psi(\tilde x)) = \set{\tilde x}$ then implies that $ \tilde  U \to U'$ is an isomorphism, as claimed. 

The image of $\pi^{-1}(Z_{\rm{sm}}) \cap \wt U$ in $U'$ is a dense open subset of $U'$. Replace $U'$ by this image, so that $\tilde U = \Psi^{-1}(U')$ is contained in $\pi^{-1}(Z_{\rm{sm}})$. In particular, we have $U = \pi(\tilde U) \subset Z_{\rm{sm}}$, which implies that $\tilde U \to U$ is an isomorphism, hence $U \to U'$ is an isomorphism.  Altogether, this proves item \eqref{item:moduli:second}. 

Finally, item \eqref{item:moduli:third} is a direct consequence of items \eqref{item:moduli:first} and \eqref{item:moduli:second}. 
\end{proof}

Examples of subvarieties $Z \subset \Agr$ that satisfy the condition on the endomorphism ring in item \eqref{item:moduli:second} of Proposition \ref{birational} are provided by the following lemma, which we record here for later use. 

\begin{lemma} \label{lemma:endo-mono}
    Let $Z \subset \ca A_{g,\delta, [r]}$ be a subvariety with maximal monodromy. Then $\End(\ca X_{\bar \eta}) = \Z$, where $f \colon \ca X \to \ca A_{g,\delta,[r]}$ denotes the universal family and $\eta \in Z$ the generic point. 
\end{lemma}
\begin{proof}
Consider the canonical variation of Hodge structure with underlying local system local system $R^1f_\ast \QQ$ on $\ca A_{g,\delta,[r]}$. Any subvariety  $Z \subsetneq \ca A_{g,\delta, [r]}$ contained in the locus of $t \in \ca A_{g,\delta,[r]}$ such that $\End(\ca X_t) \neq \Z$ satisfies $Z \subset S$ for some Hodge locus $S \subsetneq \ca A_{g,\delta,[r]}$ of this VHS. Such a Hodge locus $S$ is a special subvariety of $\ca A_{g,\delta,[r]}$, and hence $\pi_1^{\textnormal{\'et}}(S) \to \pi_1^{\textnormal{\'et}}(\Agdeltar)$ is not surjective. 
\end{proof}

\section{Hecke translates of images of subvarieties of $\ca A_g$ in $\ca A_{ng}$}

\label{section:hecke}

By the construction of Section \ref{section:morphisms}, to any symmetric positive definite $\alpha \in \rm{M}_n(\ZZ)$, one can associate a finite morphism $\Phi \colon \Agr \to \ca A_{ng, \delta, [r]}$. By Section \ref{sec:correspondences}, to any $a \in \GSp_{2ng}(\QQ)_+$ one can associate a Hecke correspondence $\tau_a$ between $\ca A_{ng, \delta, [r]}$ and $\ca A_{ng, [r]}$. For a curve in a surface $C \subset S \subset \Agr$, one thus obtains a nested sequence of subschemes $\tau_a(\Phi(C)) \subset \tau_a(\Phi(S)) \subset \ca A_{ng, [r]}$.  
The goal of this section is to prove Theorem \ref{theorem:degree} below, which roughly says that when the curve $C$ and the surface $S$ are proper and sufficiently generic, the ratio $\deg(C)/\deg(S)$ does not change when applying $\Phi$ and $\tau_a$ to $C \subset S$. 

\subsection{Hecke correspondences and degrees of curves and surfaces.} \label{subsec:hecke-degrees}For a proper variety $X$ of dimension $d$ and $L_1, \dotsc, L_d \in \Pic(X)$, consider the intersection number $(L_1 \cdots L_d) \in \ZZ$, see  \cite[Chapter VI, Definition 2.6]{kollar-rationalcurves} or  \cite[Definition 1.7]{debarre}. Hodge line bundles, and degrees of proper subvarieties of moduli spaces of abelian varieties, are defined as follows. 

\begin{definition} \label{def:hodgedegree}
For a moduli space of abelian varieties $\ca A_{g, \delta, [r]}$, let $L_{g, \delta} = \det(f_\ast \Omega_{\ca X/\ca A_{g,\delta, [r]}})$ be the Hodge line bundle on $\ca A_{g,\delta, [r]}$, where $f \colon \ca X \to \ca A_{g,\delta, [r]}$ denotes the universal family. We put $L_g = L_{g, (1, \dotsc, 1)} \in \Pic(\Agr)$. 
For an irreducible subvariety $Z \subset \ca A_{g,\delta, [r]}$ which is proper and of dimension $d$, we define the \emph{degree} of $Z$ to be the $d$-th self intersection number $(L_{g,\delta}|_Z)^d$ of the pull-back $L_{g,\delta}|_Z \in \Pic(Z)$ of the Hodge line bundle $L_{g,\delta}$ to $Z$. 
\end{definition}

Let $n,g \geq 1$ and $r \geq 3$ be integers. Let $\alpha \in \rm{M}_n(\ZZ)$ be symmetric and positive definite, and let $a \in \GSp_{2ng}(\QQ)_+$. 
Combining the constructions in Sections \ref{sec:correspondences} and \ref{section:morphisms}, we obtain the following diagram, in which the maps $p_1$ and $p_2$ are finite étale, and the morphism $\Phi$ is finite (cf.\ Proposition \ref{birational}):
\begin{align}\label{correspondence:specialcase}
\begin{split}
\xymatrix{
&{\ca A}_a\ar[dl]_-{p_1} \ar[dr]^-{p_2}&& \\
\ca A_{ng, [r]} && \ca A_{ng,\delta, [r]}& \ca A_{g, [r]}.  \ar[l]_-\Phi
}
\end{split}
\end{align}
We can now state the main result of this section, which is one of the main technical ingredients in our proof of Theorem \ref{theorem:generic}.

\begin{theorem}\label{theorem:degree}
Let $g, n \geq 1$ and $r \geq 3$. Let $\alpha \in \rm{M}_n(\ZZ)$ be symmetric and positive definite,  and $a \in \GSp_{2ng}(\QQ)_+$. 
Let $\tilde  C \subset \tilde  S$ be a smooth curve in a smooth proper surface, and $\tilde  S \to S$ a proper birational morphism onto a surface $S \subset \ca A_{g,[r]}$. Let $C \subset S$ be the image of $\tilde  C$, and $C' = \Phi(C)$. Assume:
\begin{enumerate}
    \item 
The morphism of curves $\tilde  C \to C$ is birational. 
\item The maps 
$\pi_1^{\textnormal{\'et}}(\tilde C) \to \pi_1^{\textnormal{\'et}}(\tilde S)$ and $\pi_1^{\textnormal{\'et}}(\tilde  S) \to \pi_1^{\textnormal{\'et}}(\Agr)$ are surjective. 
\item 
    There is no irreducible curve $D \subset S$ with $D\neq C$ so that $\tau_a(C')$ and $\tau_a(D')$ share an irreducible component, where $D' = \Phi(D)$.  
    \item The integer $r \geq 3$ is a prime number. 
    \end{enumerate}
    Then for each irreducible component $C^\# \subset \tau_a(C')$ 
    there exists an irreducible component $S^\# \subset \tau_a(S')$ with $C^\# \subset S^\#$, such that 
    $$
    \frac{\deg(C^\#)}{\deg(S^\#)} = \frac{\deg(C')}{\deg(S')} = \frac{\deg(C)}{\deg(S)}. 
    $$
\end{theorem}

The rest of Section \ref{section:hecke} is devoted to the proof of Theorem \ref{theorem:degree}. 

\subsection{Étale covers of curves in surfaces with maximal monodromy} \label{section:etalemonodromy}
Let $g \geq 1, r \geq 3$ be integers.  
Let $\tilde  C \subset \tilde  S$ be a smooth curve in a smooth proper surface. Let $\tilde  S \to S$ be a proper birational morphism onto a surface $S \subset \ca A_{g,[r]}$, let $C \subset S$ be the image of $\tilde  C$, and assume $\tilde  C \to C$ is birational. Assume moreover that $\pi_1^{\textnormal{\'et}}(\tilde C) \to \pi_1^{\textnormal{\'et}}(\tilde S)$ and $\pi_1^{\textnormal{\'et}}(\tilde  S) \to \pi_1^{\textnormal{\'et}}(\Agr)$ are surjective.
    
\begin{lemma}\label{irred-hecke}
 Let $p \colon \ca A \to \mathcal A_{g,[r]}$ be a finite \'etale cover. Let $R = \sqcup_{i=1}^n R_i$ (resp.\ $B = \sqcup_{j=1}^m B_j$), with $B \subset R \subset \ca A$, be the decomposition of the inverse image of $ S$ (resp.\ $C$) into connected components. Then these connected components are irreducible, we have 
 $n=m$, and there is exists a permutation $\sigma \in \mathfrak S_n$ such that for each $i$, we have $R_{\sigma(i)}|_{ C} = B_i$ as étale covers of $C$.  
 \end{lemma}
 \begin{proof}
 To prove the lemma, we may assume that $\ca A$ is connected. It then suffices to show that the pull-backs of $\ca A$ to $S$ and $C$, which are étale covers $R \to S$ and $B \to C$, are irreducible. Pulling these covers further back to $\wt S$ and $\wt C$, we obtain étale covers $\wt R \to \wt S$ and $\wt B \to \wt C$. These are connected by the assumptions on monodromy, and hence irreducible because $\wt C$ and $\wt S$ are smooth. Thus $R$ and $B$ are irreducible. 
  \end{proof}

We continue with the above notation. Let $n \geq 1$ be an integer, and let $\alpha \in \rm{M}_n(\ZZ)$ be symmetric and positive definite matrix. Let $a \in \GSp_{2ng}(\QQ)_+$, and consider diagram \eqref{correspondence:specialcase}. Define
$$
C' = \Phi(C) \qquad \text{and} \qquad S' = \Phi(S). 
$$
Then $
 C' \subset S' \subset \ca A_{ng, \delta, [r]}$, and we define a one-dimensional (resp.\ two-dimensional) subscheme $B'$ (resp.\ $R'$) of ${\ca A}_a$ as follows:
\begin{align}
\label{align:C'tildeC'S}
\begin{split}
B' \subset R' \subset {\ca A}_a \qquad B' = p_2^{-1}(C'), \qquad  R' = p_2^{-1}(S').
\end{split}
\end{align}

\begin{lemma} \label{lemma:C''S''}
The following assertions are true. 
\begin{enumerate}
\item 
The varieties $B'$ and $R'$ have the same number of irreducible components. 
\item Each irreducible component $B_i' \subset B'$ of the curve $B'$ is contained in a unique irreducible component $R_i' \subset R'$ of the surface $R'$.  
\item We have $B_i' = R_i'|_{C'}$ as étale covers of $C'$.  
\end{enumerate}
\end{lemma}
\begin{proof}
Let ${\ca A}$ be the fiber product of ${\ca A}_a$ and $\ca A_{g,[r]}$ over $\ca A_{ng, \delta, [r]}$. Denote by $q_2 \colon \ca A \to \Agr$ the finite étale cover of $\Agr$ induced by $p_2 \colon \ca A_a \to \ca A_{ng, \delta, [r]}$. 
We let $B \subset R \subset \ca A$ be the inverse image of $C \subset S \subset \Agr$ in $\ca A$. This gives finite étale covers
\begin{align*}
B \to C, \qquad R \to S.
\end{align*}
By Lemma \ref{irred-hecke}, we have decompositions  into connected components
\begin{align} \label{star}
B = \sqcup_{i = 1}^m B_i, \quad R = \sqcup_{i = 1}^m R_i, \quad \text{with} \quad  B_i \subset R_i \quad  (i \in \set{1, \dotsc, m}),
\end{align}
and the schemes $B_i$ and $R_i$ are irreducible for each $i$. 

Remark that in the following diagram, each square is cartesian:
\begin{align}\label{pullback-etale-Ag}
\begin{split}
\xymatrix{
 & {\ca A} \ar[dd]^{\smash{\lower15pt\hbox{$\scriptstyle q_2$}}} \ar[dl]_-{\Psi} && R \ar[dl] \ar@{_{(}->}[ll]  \ar[dd] && B \ar[dl]\ar[dd] \ar@{_{(}->}[ll]  \\
{\ca A}_a  \ar[dd]^-{p_2} &&R' \ar@{_{(}->}[ll] \ar[dd]&&B' \ar@{_{(}->}[ll] \ar[dd]& \\
 &\ca A_{g,[r]} \ar[dl]_-{\Phi}&   &S \ar@{_{(}->}[ll] \ar[dl] &  &C \ar[dl]\ar@{_{(}->}[ll]& \\
\ca A_{ng, \delta, [r]} &&S' \ar@{_{(}->}[ll]&&C'.  \ar@{_{(}->}[ll]&
}
\end{split}
\end{align}
As $\pi_1^{\textnormal{\'et}}(\tilde C) \to \pi_1^{\textnormal{\'et}}(\tilde S)$ and $\pi_1^{\textnormal{\'et}}(\tilde S) \to \pi_1^{\textnormal{\'et}}(\Agr)$ are surjective, the curve $C$ and the surface $S$ have maximal monodromy. Consequently, by 
Proposition \ref{birational} and Lemma \ref{lemma:endo-mono}, the morphisms $C \to C'$ and $S \to S'$ are birational, and are in fact an isomorphism over a dense open subset of the target. Let $U' \subset S'$ be a dense open subset with preimage $U = (\Phi|_S)^{-1}(U')$, such that $U \to U'$ is an isomorphism.  Let $R'_{U'} \subset R'$ and $R_U \subset R$ denote the respective inverse images in $R'$ and $R$. As the squares in \eqref{pullback-etale-Ag} are cartesian, the square
\[
\xymatrix{
R_U \ar[d] \ar[r] & R'_{U'} \ar[d] \\
U \ar[r] & U'
}
\]
is cartesian. Since $U \to U'$ is an isomorphism, the map  \begin{align} \label{iso:wtU}R_U \xrightarrow{\sim} R'_{U'}\end{align}
is an isomorphism. Moreover, as $R \to S$ is finite étale, 
the components of $R$ dominate $S$. Hence $R_U \subset R$ is dense open, so that by \eqref{star}, we have
$$
R_U = (R_U)_1 \sqcup \cdots \sqcup (R_U)_m, 
$$
with $(R_U)_i \subset R_i$ a dense open subset for each $i$. In view of \eqref{iso:wtU}, this yields a decomposition
\begin{align}\label{decomposition:wtU'}
R'_{U'} = (R'_{U'})_1 \sqcup \cdots \sqcup (R'_{U'})_m
\end{align}
of $R'_{U'}$ into connected components, and the $(R'_{U'})_i$ are irreducible.  
As $U' \subset S'$ is dense open and the irreducible components of $R'$ dominate $S'$, the inverse image $ R'_{U'} \subset R'$ of $U' \subset S'$ is dense open.  Thus, \eqref{decomposition:wtU'} implies that the decomposition of $R'$ into its irreducible components is of the form
$
R' = R_1' \cup \cdots \cup R_m'$, $R_i' = \Psi(R_i)$, with $R_i' \neq R_j'$ for all $ i \neq j$. 
An analogous argument shows that $B'$ composes into irreducible components as follows:
\begin{align} \label{wtC'-decomposition}
B' = B_1' \cup \cdots \cup B_m', \quad \quad B_i' = \Psi(B_i), \quad \quad B_i' \neq B_j' \quad \forall i \neq j. 
\end{align}
This shows in particular that $B'$ and $R'$ have the same number of irreducible components  $B_i'$ and $R_i'$, and that $B_i' \subset R_i'$ for each $i$. 

Assume $B_i' \subset R_j'$ for some $i,j \in \set{1, \dotsc, m}$. We claim that $i = j$. To prove this, consider the morphism $\Psi|_{R_i} \colon R_i \to R_i'$, and define
$
E \subset (\Psi|_{R_i})^{-1}(B_j')
$
as a one-dimensional irreducible component of $(\Psi|_{R_i})^{-1}(B_j')$. We have $$\Phi({q_2}(E)) = p_2(\Psi(E)) = p_2(B_j') = C',$$ hence ${q_2}(E) \subset \Phi^{-1}(C')$. As $\Phi^{-1}(C') = C \cup \set{x_1, \dotsc, x_k}$ for some closed points $x_i \in \Agr$ by 
item \eqref{item:moduli:third} of Proposition \ref{birational}, and as $q_2$ is finite, 
it follows that ${q_2}(E) = C$. Thus, $E$ is contained in ${q_2}^{-1}(C) = B$, hence $
E \subset B \cap R_i = B_i,
$
which implies $E = B_i$. As $E \subset (\Psi|_{R_i})^{-1}(B_j')$, we have $\Psi(E) = B_j'$, and we conclude that $B_i' = \Psi(B_i) = \Psi(E) = B_j'$. Hence $i=j$.

Finally, let $D \subset  R_i'|_{C'}$ be an arbitrary irreducible component of $R_i'|_{C'}$. Then $D \subset p_2^{-1}(C') = B'$. By \eqref{wtC'-decomposition}, this implies $D = B_j'$ for some $j$. As $D \subset R_i'$, we get $B_j' \subset R_i'$, so that $i=j$ by the above, and hence $D = B_i'$. 
\end{proof}

\subsection{Proof of Theorem \ref{theorem:degree}}
 
The goal of this subsection is to complete the proof of Theorem \ref{theorem:degree}. 
We continue with the notation of Section \ref{section:etalemonodromy}. 
Consider diagram \eqref{correspondence:specialcase}. 
As a consequence of Lemma \ref{lemma:C''S''}, we can write
\begin{align}\label{decomp}
R' = R'_1 \cup \cdots \cup R'_m, \quad \quad B' = B'_1 \cup \cdots \cup B_m',
\end{align}
with $B_i' \subset R_i'$ for each $i \in \set{1, \dotsc, m}$, and $B_i' \neq B_j'$ and $R_i' \neq R_j'$ for $i\neq j$. For each $i$, define a curve in a surface $C_i'' \subset S_i'' \subset \ca A_{ng, [r]}$ as follows:
\begin{align*}
C_i'' = p_1(B_i')\subset  p_1(R_i') = S''_i.  
\end{align*}
Note that, by definition of $\tau_a$ (see Definition \ref{tauK}), we have by \eqref{decomp} that
\[
\tau_a(C') = C''_1 \cup \cdots \cup C''_m, \quad \quad \tau_a(S') = S''_1 \cup \cdots \cup S_m'''.
\]
The varieties $C_i''$ and $S_i''$ are irreducible for each $i$, but we do not claim that the curves $C_i''$ (resp.\ the surfaces $S_i''$) are pairwise distinct.

Theorem \ref{theorem:degree} is a direct consequence of the following proposition. 
\begin{proposition}\label{ratiodegree}
Let $i \in \set{1, \dotsc, m}$. 
    Assume there is no irreducible curve $D \subset S$ with $D\neq C$ and $C_i'' \subset \tau_a(D')$, where $D' = \Phi(D)$. Then, we have 
    \begin{align} \label{degCdegS}
    \frac{\deg(C_i'')}{\deg(S_i'')} = \frac{\deg(C')}{\deg(S')} = \frac{\deg(C)}{\deg(S)}.
    \end{align}
\end{proposition}
\begin{proof}
Denote the pull-backs of the Hodge line bundles (cf.\ Definition \ref{def:hodgedegree}) to the surfaces $S$ and $S'$ by $M  = L_g|_S \in \Pic(S)$ and $ M'  = L_{ng, \delta}|_{S'} \in \Pic(S')$. Consider the morphism $\Phi|_S \colon S \to S'$. By item \eqref{flatbasechange:1} of Lemma \ref{lemma:hodgebundle}, we have $\Phi^\ast(L_{ng, \delta}) = L_g^{\otimes n} \in \Pic(\ca A_{g, [r]})$. Therefore, $(\Phi|_S)^\ast(M') = M^{\otimes n} \in \Pic(S).$
The projection formula (cf.\ \cite[Proposition 1.10]{debarre}) applied to  $\Phi|_S \colon S \to S'$ thus gives
$
n^2 \cdot \deg(S) = 
\left(\Phi|_S^\ast(M') \cdot \Phi|_S^\ast(M') \right)  
=
\left(M' \cdot M' \right) = \deg(S'),
$
where we used the fact that $\Phi|_S \colon S \to S'$ is birational (cf.\ Proposition \ref{birational} and Lemma \ref{lemma:endo-mono}) in the second equality. For analogous reasons, we have 
$
n^2 \cdot \deg(C) = \deg(C').
$
The second equality in \eqref{degCdegS} follows then by combining the two equalities $n^2 \cdot \deg(S) = \deg(S')$ and $n^2 \cdot \deg(C) = \deg(C')$. 

It remains to prove the first equality in \eqref{degCdegS}. Define 
\begin{align*}
    d_i = \deg(R_i' \to S'), \quad \quad e_i = \deg(R_i' \to S_i'').
\end{align*}
By Lemma \ref{lemma:C''S''}, we have that $B_i' = R_i' \ \times_{S'} C'$, hence 
\begin{align} \label{d_i}
 d_i = \deg(B_i' \to C'). 
\end{align}
We claim that $B_i'$ is equal to the fiber product $R_i'\times_{S_i''} C_i'' $. To see this, observe first that $B_i'$ is contained in and therefore an irreducible component of this fiber product. Letting $E'$ denote any irreducible component of this fiber product, we need to prove $E = B_i'$. Define $D' = p_2(E')$. As $E' \subset R_i'$, we have
$D' = p_2(E') \subset p_2(R_i') = S' = \Phi(S)$. 
Therefore, there exists a curve $D \subset S$ such that $\Phi(D) = D'$. If $C \neq D$, then there is an irreducible curve $D \subset S$ with $D \neq C$ and $C_i'' \subset \tau_a(D')$, 
which contradicts our hypothesis. Hence $C = D$, so that $E' \subset B'$, hence $E' = B_j'$ for some $j$. Consequently, $B_j' = E' \subset R_i'$, hence $i=j$ by the second item of Lemma \ref{lemma:C''S''}, so that $E = B_i'$. 

Consequently, 
$
B_i' = R_i'\times_{S_i''} C_i'' $, and therefore \begin{align}\label{e_i}e_i = \deg(B_i' \to C_i'').\end{align}
By Lemma \ref{quadrilateral} and the projection formula as in \cite[Proposition 1.10]{debarre}, we have
\begin{align} \label{fewequalities:S}
\begin{split}
d_i \cdot \deg(S') = \deg(R_i') = 
 e_i \cdot \deg(S_i''), 
\end{split}
\\
\begin{split}
\label{fewequalities:C}
\deg(B_i' \to C') \cdot \deg(C') = \deg(B_i') = \deg(B_i' \to C_i'') \cdot \deg(C_i''). 
\end{split}
\end{align}
Together, \eqref{d_i}, \eqref{e_i} and \eqref{fewequalities:C} imply:
\begin{align} \label{equality:final}
 d_i \cdot \deg(C') = 
\deg(B_i') = e_i\cdot \deg(C_i''). 
\end{align}
From \eqref{fewequalities:S} and \eqref{equality:final}, it follows that
\[
\frac{\deg(S_i'')}{\deg(S')} = \frac{d_i}{e_i} = \frac{\deg(C_i'')}{\deg(C')}.
\]
This proves the first equality in \eqref{degCdegS}, and we are done. 
\end{proof}

\section{Hecke translates of a curve in a surface avoiding that surface}

In this section, we find a curve in a surface whose non-trivial Hecke translates avoid that surface, and deduce consequences thereof. Together with the main result of the previous section (cf.\ Theorem \ref{theorem:degree}), this analysis will allow us to prove Theorem \ref{theorem:generic} in the next section. 

\begin{proposition}\label{proposition:avoidsurface}
Let $g \geq 2$, $r \geq 3$. Let $S\subset \Agr$ be a proper surface defined over $\smash{\Qbar}$ and $\pi \colon \tilde S \to S$ a projective birational morphism with $\tilde S$ smooth, $\tilde S$ and $\pi$ defined over $\Qbar$, 
such that $\pi_1^{\textnormal{\'et}}(\tilde S) \to \pi_1^{\textnormal{\'et}}(\Agr)$ is surjective. Then for any $d \geq 1$, there is a smooth irreducible curve $\tilde C \subset \tilde S$ defined over $\Qbar$ 
with image $C \subset S$, such that $\tilde C \to C$ is birational,  $\pi_1^{\et}(\tilde C) \to \pi_1^\et(\tilde S)$ surjective, 
$\deg(C) \geq d$, and  $\tau(C) \not\subset S$ for every non-trivial Hecke correspondence $\tau$. 
\end{proposition}
\begin{proof}
The morphism $\pi$ is the blow-up $\pi \colon \tilde S = \rm{Bl}_Z(S) \to S$ of $S$ in a closed subscheme $Z \subset S$. 
Let $L \in \Pic(S)$ be the pull-back of the Hodge line bundle. Then $L$ is ample, and for $m\gg 0$, the line bundle $L' \coloneqq \pi^\ast(L^{\otimes m}) \otimes \OO_{\tilde S}(-E)$ on $\tilde S$ is ample, where $E \subset \tilde S$ denotes the exceptional divisor.

Let $x\in S(K)$ be any point. As $\tau(x)$ equidistributes in $\Agr$ for any sequence of Hecke correspondences $\tau$ with growing degree, it follows that $\tau(x) \subset S$ for at most finitely many Hecke correspondences $\tau$. Let $I$ denote the finite set of non-trivial Hecke correspondences $\tau_i$ that satisfy $\tau_i(x) \subset S$. 

For each $\tau_i \in I$, we will now find a point $x_i \in S$ such that $\tau_i(x_i) \cap S = \emptyset$. Let $\tau_i^\vee$ denote the Hecke correspondence dual to $\tau_i$. Suppose a point $y \in S$ has the property that there is a point $z\in S \cap \tau_i(y)$. Then, $y\in \tau_i^\vee(z) $, and so $y\in S\cap \tau^\vee_i(S)$. Therefore, any point $x_i \in S \setminus \tau_i^\vee(S)$ would satisfy $\tau_i(x) \cap S = \emptyset$. We now observe that $ \tau(S) \not\subset S$ for every non-trivial Hecke correspondence $\tau$ (again by equidistribution), and therefore the set $S \setminus \tau_i^\vee(S)$ is non-empty, giving us the required $x_i$.

By Bertini's theorem, we can find a smooth irreducible curve $\tilde C \subset \tilde S$ defined over $\Qbar$, 
such that:
\begin{enumerate}[label=(\roman*)]
    \item \label{ii} We have $(L' \cdot \tilde C) \geq md$. 
    \item \label{iii} The curve $\tilde C$ is not contained in the exceptional divisor $E \subset \tilde S$. 
        \item \label{iv}The curve $C \coloneqq \pi(\tilde C) \subset S$ contains the set $\{x\} \cup \{x_i\}_{i\in I} \in S(K)$.
\end{enumerate}
By \ref{iii}, the induced map $\rho \colon \tilde C \to C$ is birational, so that 
\begin{align*}
    \left(L' \cdot \tilde C \right) &= 
    \left( 
 \left(   \pi^\ast(L^{\otimes m}) \otimes \OO_{\tilde S}(-E) \right) \cdot \tilde C 
    \right) \\ 
 & = m\deg_{\tilde C}\left( \rho^\ast(L|_C) \right) - ( E \cdot \tilde C) \\
 &  = m\deg_C\left(L|_C\right) - \left( E \cdot \tilde C \right) \\
    & \leq m \left( L \cdot C \right).
\end{align*}
Here, we used that $(E \cdot \tilde C) \geq 0$ (see \ref{iii}) in the inequality on the last line. In view of \ref{ii}, it follows that $m( L \cdot C ) \geq (L' \cdot \tilde C) \geq md$, hence $(L \cdot C) \geq d$. 
Consequently, $ \deg(C) =(L \cdot C) \geq d$. 
By the Lefschetz hyperplane theorem, the map $\pi_1^{\et}(\tilde C) \to \pi_1^\et(\tilde S)$ is surjective. As $C$ contains $x$, we have $\tau(C) \not\subset S$ for every $\tau \neq \tau_i$, $i\in I$. As $C$ contains $x_i$, we have $\tau_i(C) \not \subset S$ for $\tau_i$. 
\end{proof}

\begin{corollary}\label{singlecurve-upgraded}
Let $g,n \geq 1, r \geq 3$. 
Let $\alpha \in \rm{M}_n(\ZZ)$ be symmetric and positive definite, and let $a \in \GSp_{2ng}(\QQ)_+$. Let $C\subset S \subset \Agr$ be as in Proposition \ref{proposition:avoidsurface}. Then there is no irreducible curve $D \neq C$ with $D\subset S$, such that $\tau_a(\Phi(C))$ and $ \tau_a(\Phi(D)) \subset \ca A_{ng, [r]}$ share an irreducible component. 
\end{corollary}
\begin{proof}
Suppose that this was the case. Let $A \to C$ and $B \to D$ be the natural principally polarized abelian schemes. Then 
the geometric generic fibers of the abelian schemes $A^n \to C$ and $B^n \to D$, and hence of the abelian schemes $A \to C$ and $B\to D$, are isogenous; as $C$ has maximal monodromy, the polarization on $A \to C$ is unique (cf.\ Lemma \ref{lemma:endo-mono}), 
so the latter isogeny must 
respect polarizations up to scale. In particular, there is a Hecke operator $\tau$ on $\Agr$ such that $D \subset \tau(C)$. Note that $\tau(C)$ is irreducible because $\tilde C$ is smooth and irreducible and $\pi_1^{\textnormal{\'et}}(\tilde C) \to \pi_1^{\textnormal{\'et}}(\Agr)$ is surjective. Thus, $\tau(C) = D\subset S$, which contradicts Proposition \ref{proposition:avoidsurface}.   
\end{proof}

\section{Powers of abelian schemes over generic curves}

The goal of this section is to prove our second main theorem.

\begin{proof}[Proof of Theorem \ref{theorem:generic}]

Let $d \geq 4$ and $g=2^d$. Fix $N \geq 1$.  
Let $r \geq 3$ be a prime number. 
Let $C_{\rm{sp}} \subset \ca A_{g,[r]}$ be a Shimura curve of Mumford type; such a curve exists in view of Remark \ref{remark}. 

We claim that there exists a projective surface $S \subset \ca A_{g, [r]}$ defined over $\Qbar$, a projective birational morphism $\pi \colon \tilde S \to S$ with $\tilde S$ smooth, $\tilde S$ and $\pi$ defined over $\Qbar$, 
such that $C_{\rm{sp}} \subset S$ and such that the map $\pi_1^\et(\tilde S) \to \pi_1^\et(\Agr)$ is surjective. 
To prove this, we consider the Baily--Borel compactification $ \ca A_{g, [r]} \subset \smash{\overline{\ca A}_{g,[r]}}$. By repeatedly blowing up $\Agr$ and $\smash{\overline{\ca A}_{g,[r]}}$ in the singular locus of $C_{\rm{sp}}$, we obtain smooth varieties $\wt{\ca A} \subset \wt{\ca A}_{\rm{BB}}$ containing a smooth curve $\tilde C_{\rm{sp}} \subset \wt{\ca A} \subset \wt{\ca A}_{\rm{BB}}$, and compatible birational morphisms $\wt{\ca A}_{\rm{BB}} \to \overline{\ca A}_{g,[r]}$, $\wt{\ca A} \to \Agr$ and $\tilde C_{\rm{sp}} \to C_{\rm{sp}}$. 
By the Lefschetz hyperplane theorem for smooth quasi-projective varieties \cite{stratifiedmorse}, that concerns general hypersurfaces in a smooth quasi-projective variety, we can find a smooth curve $\tilde C_{\rm{mon}} \subset \wt{\ca A}$ such that $\pi_1^\et(\tilde C_{\rm{mon}}) \to \pi_1^\et(\wt{\ca A})$ is surjective and such that $\tilde C_{\rm{mon}} \cap \tilde C_{\rm{sp}} = \emptyset$. We then define $\tilde S \subset \wt{\ca A}_{\rm{BB}}$ as a general complete intersection of ample hypersurfaces containing the smooth disconnected curve $\tilde C_{\rm{sp}} \sqcup C_{\rm{mon}} \subset \wt{\ca A}$, such that $\tilde S$ is not contained in the exceptional locus of the birational morphism $\wt{\ca A}_{\rm{BB}} \to \overline{\ca A}_{g,[r]}$. Then $\tilde S$ is smooth by \cite{MR529493}. 
Moreover, $\tilde S$ is projective, and $\tilde S \subset \wt{\ca A}$ because the boundary $\wt{\ca A}_{\rm{BB}} \setminus \wt{\ca A}$ has codimension $g$. Since $\tilde S$ contains $\tilde C_{\rm{mon}}$, the map $\pi_1^\et(\tilde S) \to \pi_1^\et(\wt{\ca A})$ is surjective. As the map $\pi_1^\et(\wt{\ca A}) \to \pi_1^\et(\Agr)$ is an isomorphism, 
it follows that the map $\pi_1^\et(\tilde S) \to \pi_1^\et(\Agr)$ is  surjective. 
Define $S \subset \Agr$ as the image of $\tilde S$ under $\wt{\ca A} \to \Agr$. As $\tilde S \subset \wt{\ca A}$ is not contained in the exceptional locus, $\tilde S$ is the proper transform of $S$, hence $\tilde S \to S$ is projective and birational. This establishes the above claim. 

Our goal is to construct a curve $C \subset S$ with maximal monodromy defined over $\Qbar$, with the property that the geometric generic fiber of the induced abelian scheme $A \to C$ has no $n$-th power for $n \leq N$ isogenous to a Jacobian. This suffices to prove the theorem, because any such a curve defines a Hodge generic $\Qtbar$ point of $\Agr$, and hence of $\ca A_g$, that satisfies the conclusion of Theorem \ref{theorem:generic}. 

For each $n$ with $1 \leq n \leq N$, let $\overline L_n \in \Pic(\overline{\ca A}_{ng, [r]})$ be the Hodge line bundle on the Baily--Borel compactification $\overline{\ca A}_{ng,[r]} \supset \ca A_{ng,[r]}$. Let $k_n > 0 $ such that $\smash{\overline L_n^{\otimes k_n}} = \iota_n^\ast(\OO(1))$ for a closed embedding $\smash{\iota_n \colon \overline{\ca A}_{ng,[r]} \hookrightarrow \PP^{N_n}}$ into some projective space $\PP^{N_n}$.
Let $I_n$ be a finite set and $\overline H_{
n,i} \subset \overline{\ca A}_{ng,[r]} \subset \PP^{N_n}$ $(i \in I_n)$ be hypersurfaces such that, for $H_{n,i} = \overline H_{n,i} \cap \ca A_{ng,[r]}$, we have: \begin{align}\label{intersection}\ca T_{ng,[r]} = \bigcap_{\substack{1 \le n \le N \\ i \in I_n}} H_{n,i}.\end{align} 
Let $\deg(Z \subset \PP^{N_n})$ denote the degree of a subvariety of $\PP^{N_n}$. Then 
\begin{align} \label{align:deg:CS}
\deg(C \subset \PP^{N_n}) = k_n \cdot \deg(C) \quad \text{ and } \quad \deg(S \subset \PP^{N_n}) = k_n^2 \cdot \deg(S). 
\end{align}
Let \begin{align} \label{align:def:d}d = \max_{\substack{1 \le n \le N \\ i \in I_n}}\set{k_n \cdot \deg\left(\overline H_{n,i} \subset \PP^{N_n} \right)}.\end{align} 
By Proposition \ref{proposition:avoidsurface}, 
there exits a smooth curve $\tilde C \subset \tilde S$ defined over $\Qbar$, whose image we denote by $C \subset S$, such that $\tilde C \to C$ is birational,  $\pi_1^{\et}(\tilde C) \to \pi_1^\et(\tilde S)$ is surjective,  $\tau(C) \not\subset S$ for every non-trivial Hecke correspondence $\tau$, and \begin{align} \label{degree-strict}
\deg(C) > d \cdot \deg(S). 
\end{align}

Our goal is to show that $C$ works. 

Let $\eta \in C$ be the generic point of $C$ and let $A \to C$ be the induced abelian scheme. Assume for a contradiction that for some integer $n$ with $1 \leq n \leq N$, we have an isogeny \begin{align} \label{align:isogeny:JX}A_{\bar \eta}^n \to JX, \end{align} where $X$ is a compact type curve over $\Qtbarsmash$. Let $\lambda_{A, \bar \eta}$ (resp.\ $\lambda_{JX}$) be the principal polarization of $A_{\bar \eta}$ (resp.\ $JX$), and let $\mu$ be the polarization on $A^n_{\bar \eta}$ obtained by pulling back $\lambda_{JX}$ to $A^n_{\bar \eta}$ along \eqref{align:isogeny:JX}. As $C$ has maximal monodromy, we have $\End(A_{\bar \eta}) = \Z$, see Lemma \ref{lemma:endo-mono}. By \cite[Lemma 5.2]{dGFSchreieder-2025}, this implies that there exists a symmetric positive definite matrix $\alpha \in \rm{M}_n(\Z)$ such that $\mu = \lambda_\alpha$, see Notation \ref{notation:alpha}. Consider the morphism of moduli spaces
$$
\Phi \colon \ca A_{g,[r]} \to \ca A_{ng, \delta,[r]}, \quad \quad (A, \lambda_A, \bar \phi) \mapsto (A^n, \lambda_\alpha, \bar \phi_\alpha),
$$
where
$
\delta = (\delta_1, \dotsc, \delta_{ng})$ is the type of the polarization $\lambda_\alpha$.  
Consider the generic point $\eta \in C \subset \Agr$ of the curve $C$. 
There exists an element $a \in \GSp_{2ng}(\QQ)_+$ 
and a $\Qtbar$ point $\xi \in \ca A_{ng,[r]}$, with $(JX, \lambda_{JX})$ as underlying principally polarized abelian variety over $\Qtbar$, such that with respect to the correspondence $\ca A_{ng,[r]} \xleftarrow{p_1} {\ca A}_a \xrightarrow{p_2} \ca A_{ng,\delta,[r]}$ of Definition \ref{def:a}, we have
\begin{align*}
\xi \in \tau_a(\Phi(\eta)) = p_1(p_2^{-1}(\Phi(\eta))).
\end{align*}
We define $S' = \Phi(S)$ and $C' = \Phi(C)$, so that 
$
C' \subset S' \subset \ca A_{ng,\delta, [r]}. 
$
Let $C^\# \subset \tau_a(C')$ be the irreducible component of $\tau_a(C')$ with $\xi \in C^\#$. By Corollary \ref{singlecurve-upgraded}, there is no irreducible curve $D \neq C$ with $D\subset S$ such that $C^\# \subset \tau_a(\Phi(D))$. 
Therefore, by Theorem \ref{theorem:degree}, there exists an irreducible component $S^\# \subset \tau_a(S')$ with $C^\# \subset S^\#$, such that 
   \begin{align}\label{align:ratio:proof}
    \frac{\deg(C^\#)}{\deg(S^\#)} = \frac{\deg(C)}{\deg(S)}. 
    \end{align}
As $\xi \in \ca T_{ng,[r]}$, we have $C^\# \subset \ca T_{ng,[r]}$, hence \eqref{intersection} implies that
$
C^\# \subset H_{n, i}$ for each $ i \in I_n$. 
By construction, $S \subset \ca A_{g,[r]}$ contains a Shimura curve of Mumford type, hence by Theorem \ref{theorem:shimuracurve}, we have $
S^\# \not \subset \ca T_{ng,[r]} = \bigcap_{n, i \in I_n} H_{n,i}$. Thus,  $S^\# \not \subset H_{n,i}$ for some $i \in I_n$. Since the irreducible components of $S^\# \cap H_{n,i}$ have dimension at least one, they have dimension exactly one. In particular, the intersection $S^\# \cap H_{n,i} = S^\# \cap \overline H_{n,i}$ is transverse, so that $$\deg(S^\# \cap \overline H_{n,i} \subset \PP^{N_n}) = \deg(S^\# \subset \PP^{N_n}) \cdot \deg(\overline H_{n,i} \subset \PP^{N_n}).$$ 
As $C^\# \subset \overline H_{n,i}$, we have  
$
C^\# \subset S^\# \cap \overline H_{n, i} \subsetneq S^\#$, hence $C^\#$ is an irreducible component of the purely one-dimensional scheme $S^\# \cap \overline H_{n, i}$. Consequently, 
\begin{align*}
\begin{split}
\deg(C^\# \subset \PP^{N_n}) &\leq \deg( S^\# \cap \overline H_{n,i}  \subset \PP^{N_n}) \\
& = 
\deg(S^\# \subset \PP^{N_n}) \cdot \deg(\overline H_{n,i} \subset \PP^{N_n}),\end{split}
\end{align*}
and hence, by \eqref{align:deg:CS} and \eqref{align:def:d}, we get
\begin{align} \label{inequality:intersection}
\deg(C^\#) \leq k_n \cdot \deg(S^\#) \cdot  \deg(\overline H_{n,i} \subset \PP^{N_n}) \leq d \cdot \deg(S^\#). 
\end{align}
From \eqref{align:ratio:proof} and \eqref{inequality:intersection}, we deduce $
\deg(C) \leq d\cdot \deg(S)$, which contradicts \eqref{degree-strict}. Thus, 
the $n$-th power $A_{\bar \eta}^n$ of 
the geometric generic fiber $A_{\bar \eta}$ of the abelian scheme $A \to C$ is not isogenous to a Jacobian, for every $n \leq N$. 
\end{proof}

\printbibliography
\end{document}